\documentclass[12pt]{article}
\usepackage{array}
\usepackage{amsthm,amsmath,amssymb,color}
\usepackage{fullpage}
\usepackage[mathlines,switch]{lineno}

\usepackage{color}
\usepackage[normalem]{ulem}
\usepackage{enumitem}
\usepackage{tikz}
\usetikzlibrary{snakes}
\tikzset{decoration={snake, amplitude=.4mm,segment length=2mm}}
\usepackage{hyperref}
\usepackage[toc,page]{appendix}

\newtheorem{thm}{Theorem}[section]

\newtheorem{lem}[thm]{Lemma}
\newtheorem{prop}[thm]{Proposition}

\theoremstyle{definition}
\newtheorem{defn}[thm]{Definition}

\theoremstyle{remark}
\newtheorem*{rmk}{Remark}
\newtheorem{obs}[thm]{Observation}

\modulolinenumbers[1] %

\begin{document}

\title{Eigenvalues and parity factors in graphs}
\author{
	Donggyu Kim\thanks{Discrete Mathematics Group, IBS, Korea, Daejeon, 34126 and Department of Mathematical Sciences, KAIST, Korea, Daejeon, 34141, donggyu@kaist.ac.kr. Research supported by the Institute for Basic Science (IBS-R029-C1).}
	\,and
	Suil O\thanks{Department of Applied Mathematics and Statistics, The State University of New York, Korea, Incheon, 21985, suil.o@sunykorea.ac.kr. Research supported by NRF-2020R1F1A1A01048226 and  NRF-2021K2A9A2A06044515}
}

\maketitle

\begin{abstract}

Let $G$ be a graph and let $g, f$ be nonnegative integer-valued functions defined on $V(G)$ such that $g(v) \le f(v)$ and $g(v) \equiv f(v) \pmod{2}$ for all $v \in V(G)$.
A $(g,f)$-parity factor of $G$ is a spanning subgraph $H$ such that for each vertex $v \in V(G)$, $g(v) \le d_H(v) \le f(v)$ and $f(v)\equiv d_H(v) \pmod{2}$.
We prove sharp upper bounds for certain eigenvalues in an $h$-edge-connected graph $G$ with given minimum degree to guarantee the existence of a $(g,f)$-parity factor; we provide graphs showing that the bounds are optimal.
This result extends the recent one of the second author~\cite{O2022}, extending the one of Gu~\cite{Gu2014}, Lu~\cite{Lu2010},  
Bollb{\'a}s, Saito, and Wormald~\cite{Bollobas1985}, and Gallai~\cite{Gallai1950}.\\


\end{abstract}

\section{Introduction}\label{sec:intro}

For a graph $G$ and non-negative integer-valued functions $g, f$ defined on the vertex set $V(G)$ %
with $g(v) \leq f(v)$ for all $v \in V(G)$, a \emph{$(g,f)$-factor} of $G$ is a spanning subgraph $H$ of $G$ such that %
for every vertex $v \in V(G)$, $g(v) \le d_H(v) \le f(v)$.
An $f$-factor is an $(f,f)$-factor.
For $g$ and $f$ with $g(v)\equiv f(v) \pmod{2}$ for every $v\in V(G)$, a \emph{$(g,f)$-parity factor} of $G$ is a $(g,f)$-factor such that $d_H(v) \equiv f(v) \pmod{2}$ for all $v \in V(G)$.
For integers $a$ and $b$, an $[a,b]$-factor of $G$ is a $(g,f)$-factor such that $g(v) = a$ and $f(v) = b$ for all $v\in V(G)$, and a $k$-factor is a $[k,k]$-factor.
For odd (or even, respectively) integers $a$ and $b$, an odd (or even, respectively) $[a,b]$-factor is a $(g,f)$-parity factor with $g(v)=a$ and $f(v)=b$ for all $v\in V(G)$.
The eigenvalues of a graph are the eigenvalues of its adjacency matrix.

In 1891, Petersen~\cite{Petersen1891} proved that every $2$-edge-connected cubic graph has a $1$-factor, and for arbitrary integers $k$ and $r$ with $0\leq k \leq r$, every $2r$-regular graph has a $2k$-factor.
Hall~\cite{Hall1948} investigated a necessary and sufficient condition for a bipartite graph to have a $1$-factor. This is called the Hall's marriage theorem.
Tutte~\cite{Tutte1947} proved a necessary and sufficient condition for a graph to guarantee the existence of a $1$-factor and extended the result to an $f$-factor~\cite{Tutte1952}.
Lov\'{a}sz~\cite{Lovasz1970a} obtained necessary and sufficient conditions for a graph to have a $(g,f)$-factor or a $(g,f)$-parity factor.
It is however hard to check equivalent conditions of the existence of factors in graph. Thus researchers tried to investigate simple conditions ensuring the existence of factors in graphs.

One approach to ensure a certain factor is to consider certain edge-connectivity in a graph, like the Petersen's theorem.
Gallai~\cite{Gallai1950} %
investigated some sufficient conditions for a regular graph with a certain edge-connectivity to have a $k$-factor, which appeared before the Tutte's $f$-factor theorem.
Bollob\'{a}s, Saito, and Wormald~\cite{Bollobas1985} slightly improved the Gallai's result.
Kano~\cite{Kano1985} proved some sufficient conditions for a graph to have a $(g,f)$-factor, which implies the result of Bollob\'{a}s, Saito, and Wormald.

\begin{thm}[\cite{Kano1985}]\label{thm:Kano}
	Let $G$ be an $h$-edge-connected graph (allowing loops and multiple edges), and let $\theta$ be a real number with $0\leq \theta \leq 1$.
	Let $g$ and $f$ be integer-valued functions on $V(G)$ such that $g(v) \leq \theta d_G(v) \leq f(v)$ for all $v\in V(G)$, and $\sum_{v\in V(G)} f(v)$ is even.
	Let $W = \{v \in V(G) : g(v) = f(v)\}$, and let $h_e$ and $h_o$ be even and odd integers in $\{h,h+1\}$, respectively.
	If one of the following holds, then $G$ has a $(g,f)$-factor.
	\begin{enumerate}[label=\rm(\alph*)]
		\item $h\theta \geq 1$ and $h(1-\theta) \geq 1$.
		\item $d_G(v)$ and $f(v)$ are even for all $v\in W$.
		\item $d_G(v)$ is even for all $v\in W$, and $h_e \theta \geq 1$ and $h_e (1-\theta) \geq 1$.
		\item $f(v)$ is even for all $v\in W$, and $h_o (1-\theta) \geq 1$.
		\item $d_G(v)$ and $f(v)$ are odd for all $v\in W$, and $h_o \theta \geq 1$.
	\end{enumerate}
\end{thm}

Theorem~\ref{thm:Kano} was originally stated for graphs with multiple edges and without loops in~\cite{Kano1985}, but it also holds for graphs allowing multiple edges and loops.

The preceding theorem deduces a sufficient condition for a graph $G$ to have a $(g,f)$-parity factor, where $\sum_{v\in V(G)} f(v)$ is even.
Let $G'$ be the graph obtained from $G$ by attaching $\frac{f(v)-g(v)}{2}$ loops for each vertex $v\in V(G)$, where each loop at $v\in V(G')$ contributes~$2$ to the degree at $v$.
Then $G$ has a $(g,f)$-parity factor if and only if $G'$ has a $f$-factor. Therefore, one may obtain a sufficient condition for $G$ to have a $(g,f)$-parity factor from Theorem~\ref{thm:Kano}; see Chen~{\cite[Lemma~5.2]{Chen1996}}.

Another approach is to bound eigenvalues of graphs in order to ensure the existence of certain factors in graphs.
Brouwer and Haemers~\cite{Brouwer2005} proved that if the third largest eigenvalue of a regular graph with even number of vertices is small enough in terms of the degree, then it has a $1$-factor.
Cioab\u{a}, Gregory, and Haemers~\cite{Cioaba2009} improved the bound, and Lu~\cite{Lu2010} extended this result to a $k$-factor.
Gu~\cite{Gu2014} proved that an $h$-edge-connected $r$-regular graph has a $k$-factor if a certain  eigenvalue is small in terms of $r$ and $h$. Very recently, O~\cite{O2022} found an eigenvalue condition for an $h$-edge-connected $r$-regular graph to have an odd or even $[a,b]$-factor.

We investigate sufficient conditions for a graph to have a $(g,f)$-parity factor in terms of the minimum degree, edge-connectivity, and eigenvalues. This extends the results of Kano~\cite{Kano1985} and O~\cite{O2022}.
For positive integers $r$ and $\eta$ with $r>\eta$, we define
\begin{equation*}
	\rho(r,\eta) = 
	\begin{cases}
		\frac{1}{2}(r-2+\sqrt{(r+2)^2-8\lfloor \eta/2 \rfloor}) & \textrm{if at least one of $r$ and $\eta$ is even}, \\
		\frac{1}{2}(r-3+\sqrt{(r+3)^2-4\eta}) & \textrm{if both $r$ and $\eta$ are odd with $\eta \geq 3$}, \\ 
		\mu & \text{if $r$ is odd and $\eta=1$},
	\end{cases}
\end{equation*}
where $\mu$ is the largest real root of $x^3-(r-2)x^2-2rx+r-1 = 0$.

\begin{thm}\label{thm:main}
	Let $G$ be an $h$-edge-connected graph, and let $\theta$ be a real number with $0 < \theta < 1$.
	Suppose that $g$ and $f$ are integer-valued functions on $V(G)$ such that $g(v) \leq \theta d_G(v) \leq f(v)$ and $g(v) \equiv f(v) \pmod{2}$ for all $v\in V(G)$, and $\sum_{v\in V(G)} f(v)$ is even.
	Let $\theta^* = \min\{\theta,1-\theta\}$, and $h_e$ and $h_o$ be even and odd integers in $\{h,h+1\}$, respectively.
	If one of \ref{item:m1}-\ref{item:m5} holds, then $G$ has a $(g,f)$-parity factor.
	\begin{enumerate}[label=\rm(\alph*)]
		\item\label{item:m1}
		\begin{enumerate}[label=\rm\roman*.]
		    \item $h \geq 1/\theta^*$, or
		    \item $h < 1/\theta^* \leq \delta(G)$ and $\lambda_{\left\lceil \frac{2}{1-\theta^* h} \right\rceil} (G)
				< \rho(\delta(G), \lceil 1/\theta^* \rceil -1)$.
		\end{enumerate}
		
		\item\label{item:m2} $d_G(v)$ and $f(v)$ are even for all $v\in V(G)$.
		
		\item\label{item:m3} $d_G(v)$ is even for all $v\in V(G)$, and one of the following holds.
		\begin{enumerate}[label=\rm\roman*.]
		    \item $h_e \geq 1/\theta^*$.
		    \item $h_e < 1/\theta^* \leq \delta(G)$, and $\lambda_{\left\lceil \frac{2}{1-\theta^* h_e} \right\rceil} (G)
				< \rho(\delta(G), \lceil 1/\theta^* \rceil -1)$.
		\end{enumerate}
		
		\item\label{item:m4} $f(v)$ is even for all $v\in V(G)$, and one of the following holds.
        \begin{enumerate}[label=\rm\roman*.]
            \item $h_o \geq 1/(1-\theta)$.
            \item $h_o < 1/(1-\theta) \leq \delta(G)$ and $\lambda_{\left\lceil \frac{2}{1-(1-\theta) h_o} \right\rceil} (G)
				< \rho(\delta(G), \lceil 1/(1-\theta) \rceil -1)$.
        \end{enumerate}

		\item\label{item:m5}
		$d_G(v) \equiv f(v) \pmod{2}$ for all $v\in V(G)$, and one of the following holds.
		\begin{enumerate}[label=\rm\roman*.]
		    \item $h_o \geq 1/\theta$.
		    \item $h_o < 1/\theta \leq \delta(G)$ and $\lambda_{\left\lceil \frac{2}{1-\theta h_o} \right\rceil} (G)
				< \rho(\delta(G), \lceil 1/\theta \rceil -1)$.
		\end{enumerate}
		
	\end{enumerate}
\end{thm}

Note that Theorem~\ref{thm:main}-\ref{item:m5} implies when both $d_G(v)$ and $f(v)$ are odd for every $v\in V(G)$.

Non-eigenvalue conditions~\ref{item:m1}-i, \ref{item:m2}, \ref{item:m3}-i, \ref{item:m4}-i, and \ref{item:m5}-i could be deduced from Theorem~\ref{thm:Kano} as mentioned earlier.
In Section~\ref{sec:tight}, we provide graphs showing that the eigenvalue conditions in Theorem~\ref{thm:main} are optimal.

We organize this paper as follows:
In Section~\ref{sec:prelim}, we present some terminologies from graph theory.
In Section~\ref{sec:tools}, we provide important tools like the Lov\'{a}sz's $(g,f)$-parity factor theorem and the (Quotient) Interlacing Theorem. Also lower bounds of eigenvalues of graphs with certain number of vertices and edges are determined, and graphs holding equalities in the bounds are provided.
Finally, we prove Theorem~\ref{thm:main} in Section~\ref{sec:main} and that the lower bounds of the eigenvalues in Theorem~\ref{thm:main} are best possible in Section~\ref{sec:tight}.

\section{Preliminaries}\label{sec:prelim}

For a positive integer $k$, let $[k] := \{1,\dots,k\}$.
For an $X\times Y$ matrix $A$, $X' \subseteq X$, and $Y' \subseteq Y$, let $A[X',Y']$ be the submatrix of $A$ indexed by $X'$ and $Y'$.
When $X = Y$ and $X' = Y'$, we may write $A[X']$ to denote $A[X',X']$.
For a real $x$, $\lfloor x \rfloor$ (or $\lceil x \rceil$ respectively) is the largest integer less than or equal to $x$ (or the smallest integer greater than or equal to $x$, respectively).

In this paper, every graph does not allow loops or parallel edges unless mentioned.
For a graph $G$ and a vertex subset $S$, $G[S]$ is the subgraph of $G$ induced by $S$.
We denote the degree of a vertex $v$ of $G$ by $d_G(v)$.
Let $d_G(S) := \sum_{v\in S} d_G(v)$.
For an integer-valued function~$f$ on $V(G)$, let $f(S) := \sum_{v\in S} f(v)$.
The \emph{complement} of $G$, denoted by $\overline{G}$, is a graph on $V(G)$ such that two distinct vertices $v$ and $w$ are adjacent in $\overline{G}$ if and only if they are not adjacent in $G$.
We denote the minimum degree and the maximum degree of $G$ by $\delta(G)$ and $\Delta(G)$, respectively.
For disjoint subsets $S$ and $T$ of $V(G)$, let $[S,T]_G$ be the set of edges whose one end is in $S$ and the other end is in $T$.
We may omit the subscript $G$ if no confusion arises.
We denote by $\kappa'(G)$ the edge-connectivity of $G$.

The adjacency matrix $A(G)$ of a graph $G$ is a $V(G)\times V(G)$ symmetric $\{0,1\}$-matrix, %
whose $(v,w)$-entry is $1$ if and only if $v$ and $w$ are adjacent in $G$.
The eigenvalues of~$G$ are the eigenvalues of $A(G)$.
For an $n\times n$ square matrix $A$ and an integer $i\in\{1,2,\dots,n\}$, we denote by $\lambda_i(A)$  the $i$-th largest eigenvalue of~$A$.
Namely, $\lambda_1(A) \geq \dots \geq \lambda_n(A)$ are all eigenvalues of $A$.
For $i \in [|V(G)|]$, we similarly denote by $\lambda_i(G)$ the $i$-th largest eigenvalue of~$G$.
Note that $\lambda_1(G)=\max_{i \in [n]}|\lambda_i(G)|$, the \emph{spectral radius} of $G$.

For a graph $G$ and a partition $\pi = \{W_1,\dots,W_m\}$ of $V(G)$, the \emph{quotient matrix} of $G$ with respect to $\pi$, denoted by $G/\pi$, is a $\pi \times \pi$ matrix whose $(W_i,W_j)$-entry is $|[W_i,W_j]_G| / |W_i|$.
We call $\pi$ an \emph{equitable} partition of $G$ if for each $i,j \in [m]$ (possibly, $i=j$), the number of neighbors of $v\in W_i$ in $W_j$ is constant, regardless of the choice of $v\in W_i$.
Equivalently, $\pi$ is equitable if and only if $|[\{v\},W_j]_G| = |[W_i,W_j]_G|/|W_j|$ for all $i,j\in[n]$ and $v\in W_i$.

For a graph $G$ and a positive integer $m$, we denote by $mG$ the disjoint union of $m$ copies of $G$.
For disjoint graphs $G_1, G_2, \dots, G_t$ with $t\geq 2$, let $G_1 \vee G_2 \vee \dots \vee G_t$ be a graph obtained from the union of $G_1, \dots, G_t$ by adding all possible edges between $G_i$ and $G_{i+1}$ for each $i\in[t-1]$.

\section{Tools}\label{sec:tools}

We first introduce two well-known theorems about parity factors and eigenvalues.

\begin{thm}[Lov\'{a}sz's $(g,f)$-parity factor theorem~\cite{Lovasz1970}; see Theorem~6.1 in~\cite{Akiyama2011}]
	\label{thm:parityfactor}
	Let $G$ be a graph on the vertex set $V$, and $g$ and $f$ be integer-valued functions on $V$ such that $g(v) \leq f(v)$ and $g(v) \equiv f(v) \pmod{2}$ for each $v\in V$.
	Then $G$ has a $(g,f)$-parity factor if and only if for all disjoint subsets $S$ and $T$ of $V$, it holds that
	\begin{equation*}
	  d_G(T) - g(T) + f(S) - |[S,T]_G| - q(S,T) \geq 0,
	\end{equation*}
	where $q(S,T)$ is the number of components $C$ in $G \setminus (S\cup T)$ such that $f(V(C)) + |[T,V(C)]_G|$ is odd.
\end{thm}

\begin{thm}[Interlacing theorem; see~{\cite[Theorem~2.5.1.(i)]{Brouwer2012}}]
	\label{thm:interlacing}
	Let $A$ be an $n\times n$ real symmetric matrix, and $S$ be an $m\times n$ real matrix such that $S^T S = I_m$ with $m<n$.
	Let $\lambda_1 \geq \dots \geq \lambda_n$ be the eigenvalues of $A$, and $\mu_1 \geq \dots \geq \mu_m$ be the eigenvalues of $B = S^T A S$.
	Then
	\begin{equation*}
		\lambda_i \geq \mu_i \geq \lambda_{n-m+i}
	\end{equation*}
	for each $i\in[m]$.
\end{thm}

The following two propositions are useful to compute and bound the spectral radius of a graph.
We will use two propositions without mentioning them.

\begin{prop}[see {\cite[Chapters~8 and~9]{Godsil2001}}]\label{prop:equitable}
	For a graph $G$ and an equitable partition $\pi$ of~$G$, $\lambda_1(G) = \lambda_1(G/\pi)$.
\end{prop}

\begin{prop}[\cite{Collatz1957}]
	The spectral radius of a graph is at least the average degree of the graph.
\end{prop}

We now provide a lower bound for the spectral radius of a graph in terms of certain number of vertices and edges, which will be used in the proof of Theorem~\ref{thm:main} to obtain the eigenvalue conditions.

\begin{lem}\label{lem:extreme}
	Let $r$ and $\eta$ be positive integers with $r > \eta$.
	Let $H$ be a graph with at least $r+1$ vertices and at least $(r|V(H)| - \eta)/2$ edges.
	Then the following hold.
	\begin{enumerate}[label=\rm(\roman*)]
	    \item\label{item:ext1} If at least one of $r$ and $\eta$ is even, then $\lambda_1(H) \geq \frac{r-2 + \sqrt{(r+2)^2 - 8 \lfloor \eta/2 \rfloor}}{2}$.
	    \item\label{item:ext2} If both $r$ and $\eta$ are odd, then $\lambda_1(H) \geq \frac{r-3 + \sqrt{(r+3)^2 - 4\eta}}{2}$.
	    \item\label{item:ext3} If $r$ is odd and $\eta = 1$ and every vertex of $H$ except one has degree at least $r$, then $\lambda_1(H) \geq \mu$, where $\mu$ is the largest real root of $x^3 - (r-2)x^2 - 2rx + r - 1 = 0$.
	\end{enumerate}
\end{lem}

\begin{proof}
	\ref{item:ext1}
	Suppose that $|V(H)| \geq r+2$.
	If $r$ is even and $\eta$ is odd, then $2|E(H)| \geq r(r+2) - \eta + 1 = r(r+2) - 2 \lfloor \eta/2 \rfloor$ by parity.
	If $\eta$ is even, then $\eta = 2 \lfloor \eta/2 \rfloor$.
	Hence $2|E(H)| \geq r(r+2) - 2 \lfloor \eta/2 \rfloor$, so
	\begin{equation*}
	  \lambda_1(H) \geq \frac{2|E(H)|}{|V(H)|} \geq \frac{r(r+2) - 2\lfloor \eta/2 \rfloor}{r+2}
	  > \frac{r-2 + \sqrt{(r+2)^2 - 8 \lfloor \eta/2 \rfloor}}{2}.
	\end{equation*}
	Therefore, we may assume that $|V(H)| = r+1$.
	Now we suppose that $2|E(H)| \geq r(r+1) - 2 \lfloor \eta/2 \rfloor + 1$.
	Then
	\begin{equation*}
		\lambda_1(H) \geq \frac{2|E(H)|}{|V(H)|} \geq \frac{r(r+1) - 2\lfloor \eta/2 \rfloor + 1}{r+1}
		> \frac{r-2 + \sqrt{(r+2)^2 - 8 \lfloor \eta/2 \rfloor}}{2}.
	\end{equation*}
	Therefore, we may assume that $2|E(H)| = r(r+1) - 2 \lfloor \eta/2 \rfloor$.

	The number of vertices of degree $r$ in $H$ is at least $r+1-2\lfloor \eta/2 \rfloor > 0$.
	Let $A$ be a set of $r+1-2\lfloor \eta/2 \rfloor$ vertices of degree $r$ in $H$, and let $B = V(H) \setminus A$.
	If $\eta \leq 1$, then $H$ is the complete graph with $r+1$ vertices, which implies that $\lambda_1(H) = r$.
	Hence we may assume that $\eta \geq 2$, so $B \neq \emptyset$.
	Let $Q$ be the quotient matrix of $H$ with respect to a partition $\{A,B\}$ of $V(H)$.
	One may check that
	\begin{equation*}
	  Q = \begin{pmatrix}
		  r - 2\lfloor \eta/2 \rfloor		&	2\lfloor \eta/2 \rfloor \\
		  r + 1 - 2\lfloor \eta/2 \rfloor	&	2\lfloor \eta/2 \rfloor - 2
	  \end{pmatrix},
	\end{equation*}
	where the first row and column are indexed by $A$, and the second row and column are indexed by $B$.
	Moreover,
	\begin{equation*}
	  \lambda_1(Q) = \frac{r-2 + \sqrt{(r+2)^2 - 8 \lfloor \eta/2 \rfloor}}{2}.
	\end{equation*}
	By the interlacing theorem, $\lambda_1(H) \geq \lambda_1(Q)$, which completes the proof of~\ref{item:ext1}.

	\ref{item:ext2}
	For convenience, let $\nu = \frac{1}{2}(r-3 + \sqrt{(r+3)^2-4\eta})$.
	Suppose that $|V(H)| \geq r+3$.
	Then
	\begin{equation*}
		\lambda_1(H) \geq \frac{2|E(H)|}{|V(H)|} \geq \frac{r(r+3) - \eta}{r+3}
		> \nu.
	\end{equation*}
	Suppose that $|V(H)| = r+1$.
	By parity, $2|E(H)| \geq r(r+1) - \eta + 1$ so
	\begin{equation*}
		\lambda_1(H) \geq \frac{2|E(H)|}{|V(H)|} \geq \frac{r(r+1) - \eta + 1}{r+1} > \nu.
	\end{equation*}
	Hence we may assume that $|V(H)| = r+2$.
	Suppose that $2|E(H)| \geq r(r+2) - \eta + 1$.
	Then
	\begin{equation*}
		\lambda_1(H) \geq \frac{2|E(H)|}{|V(H)|} \geq \frac{r(r+2) - \eta+1}{r+2}
		> \nu.
	\end{equation*}
	Therefore, we may assume that $2|E(H)| = r(r+2) - \eta$.

	Let $\alpha$ be the number of vertices of degree $r+1$ in $H$, and $\beta$ be the number of vertices of degree $r$ in $H$.
	Since $(r+1)\alpha + r\beta + (r-1)(r+2-\alpha-\beta) \geq \sum_{v\in V(H)} d_H(v) = r(r+2) - \eta$, we have $2\alpha+\beta \geq r+2-\eta > 0$.

	Suppose that the number of vertices of degree $r+1$ is at least $(r+2-\eta)/2$.
	Then we can take a partition $\{A,C\}$ of $V(H)$ such that $|A| = (r+2-\eta)/2$ and all vertices in $A$ have degree $r+1$.
	The quotient matrix $Q_1$ of $H$ with respect to the partition $\{A,C\}$ is
	\begin{equation*}
		\begin{pmatrix}
			(r-\eta)/2		&	(r+2+\eta)/2 \\
			(r+2-\eta)/2	&	(r-4+\eta)/2
		\end{pmatrix},
	\end{equation*}
	and
	\begin{equation*}
		\lambda_1(Q_1) = \frac{r-2 + \sqrt{(r+2)^2-4(\eta-1)}}{2} > \nu.
	\end{equation*}
    Therefore, we may assume that the number of vertices of degree $r+1$ is strictly less than $(r+2-\eta)/2$, so $\beta > 0$.

	Suppose that $H$ has a vertex of degree $r+1$.
	Then we can take a partition $\{A,B,C\}$ of $V(H)$ such that $2|A|+|B| = r+2-\eta$, and all vertices in $A$ has degree $r+1$, and all vertices in $B$ has degree $r$.
	Note that $|C| = r+2-|A|-|B| = \eta + |A|$.
	Let $a = |A|$, $b = |B|$, and $c = |C|$.

    For each vertex in $B$, the number of its neighbors in $C$ is either $c$ or $c-1$.
	Hence $b(c-1) \leq |[B,C]_H| \leq bc$.
	Let $d$ be the integer such that $|[B,C]_H| = bc - d$.
	Then $0 \leq d \leq b$.
	Let $Q_2$ be the quotient matrix of $H$ with respect to the partition $\{A,B,C\}$.
	Then
	\begin{equation*}
	  Q_2 = \begin{pmatrix}
		  a-1		&	b				&	c \\
		  a			&	b-2+\frac{d}{b}	&	c-\frac{d}{b} \\
		  a			&	b-\frac{d}{c}	&	c-3+\frac{d}{c}
	  \end{pmatrix},
	\end{equation*}
	and the characteristic polynomial $q_2$ of $Q_2$ is
	\begin{align*}
	  &x^3
	  - \left( a+b+c-6+ \frac{d}{c} + \frac{d}{b} \right) x^2 \\
	  &+ \left( -5a-4b-3c+11 + \frac{d}{c}(a+b+c-3) + \frac{d}{b}(a+b+c-4) \right) x \\
	  &- \left( 6a+3b+2c-6 -\frac{d}{c}(2a+b-2) - \frac{d}{b}(3a+2b+c-3) \right) \\
	  =
	  &x^3
	  - \left( r-4 + \frac{d}{bc}(r+2-a) \right) x^2 \\
	  &+ \left( -4r+3+\eta + \frac{d}{bc}(r^2+(-a+1)r-\eta-2) \right) x \\
	  &- \left( 3r-\eta+2a - \frac{d}{bc} (r^2+2r-\eta+a\eta+a) \right).
	\end{align*}
	Let $p = x^2-(r-3)x-3r+\eta$.
	Then $p(\nu) = 0$.

	We claim that $\lambda_1(Q_2) > \nu$.
	It suffices to show that $q_2(\nu)<0$.
	One can check that
	\begin{align*}
		&q_2 - \left( x+1 - \frac{d(r+2-a)}{bc} \right) p \\
		&=
		-2a + \frac{d}{bc} \left( (2r+4-\eta-3a)x - 2r^2 - 4r + \eta r + 3ar + \eta + a  \right) \\
		&=
		-2a +
		\frac{d(2r+4-\eta-3a)}{bc}
		\left(
			x - r + \frac{\eta+a}{2r+4-\eta-3a}
		\right).
	\end{align*}
	Recall that $c=\eta+a$ and $d\leq b$.
	Since $\nu < r$,
	\begin{align*}
		q_2(\nu)
		&= 
		-2a +
		\frac{d(2r+4-\eta-3a)}{bc}
		\left(
			\nu - r + \frac{\eta+a}{2r+4-\eta-3a}
		\right) \\
		&<
		-2a + \frac{d}{b}
		\leq -2a + 1 < 0
	\end{align*}
	Therefore, we may assume that $H$ has no vertex of degree $r+1$.
	
	Now, we can take a partition $\{B,C\}$ of $V(H)$ such that $|B| = r+2-\eta$ and every vertex in $B$ has degree $r$.
	As like before, $|[B,C]_H| = |B||C|-d$ for some integer $d$ with $0\leq d\leq |B|$.
	Let $Q_3$ be the quotient matrix of $H$ with respect to the partition $\{B,C\}$.
	Then
	\begin{equation*}
		Q_3 =
		\begin{pmatrix}
			r-\eta+{d}/(r+2-\eta)	&	\eta-{d}/(r+2-\eta) \\
			r+2-\eta-{d}/{\eta}		&	\eta-3+{d}/{\eta}
		\end{pmatrix},
	\end{equation*}
	and the characteristic polynomial $q_3$ of $Q_3$ is
	\begin{align*}
		x^2 - \left( r-3 + \frac{d}{(r+2-\eta)\eta} (r+2) \right) x -3r+\eta + \frac{d}{(r+2-\eta)\eta} (r^2+2r-\eta).
	\end{align*}
	Then
	\begin{align*}
		q_3(\nu)
		&= (q_3-p)(\nu)
		= \frac{d(r+2)}{(r+2-\eta)\eta} \left( -\nu + r - \frac{\eta}{r+2} \right).
	\end{align*}
	Note that $\nu > r-\frac{\eta}{r+2}$.
	Hence $q_3(\nu) \leq 0$, which implies that $\lambda_1(H) \geq \nu$.
	
	\ref{item:ext3}
	Let $p = x^3-(r-2)x^2-2rx+r-1$.
	Recall that $\mu$ is the largest real solution of~$p=0$.
	Since $r$ is odd with $r > \eta = 1$, we have $r \geq 3$.
	Suppose that $|V(H)| \geq r+4$.
	Then $\lambda_1(H) \geq \frac{2|E(H)|}{|V(H)|} \geq r - \frac{1}{r+4}$.
	One may check that
	\begin{align*}
	    p\left(r-\frac{1}{r+4}\right) &= \frac{r^3+6r^2-6r-57}{(r+4)^3} > 0,
	    \tag{1}\label{al:3-1}
	\end{align*}
	and the largest real solution $s$ of $(d/dx)p = 3x^2 - 2(r-2)x -2r$ satisfies that
	\begin{align*}
		s = \frac{2(r-2) + \sqrt{4(r-2)^2 + 24r}}{6}
		&< \frac{2(r-2) + 2(r+2)}{6} = \frac{2r}{3}
		< r- \frac{1}{r+4}.
		\tag{2}\label{al:3-2}
	\end{align*}
	The latter inequality~\eqref{al:3-2} implies that $(d/dx)p (t) > 0$ for all $t \geq r-\frac{1}{r+4}$.
	Then by~\eqref{al:3-1}, $p(t) > 0$ for all $t \geq r - \frac{1}{r+4}$, which implies that $\lambda_1(H) \geq r - \frac{1}{r+4} > \mu$.
	Thus, we may assume that $|V(H)| \leq r+3$.
	
	Suppose that $|V(H)| = r+3$ (or $r+1$).
	By parity, $2|E(H)| \geq r(r+3)$ (respectively, $r(r+1)$) and thus $\lambda_1(H) \geq 2|E(H)|/|V(H)| = r > r-\frac{1}{r+4} > \mu$.
	Therefore, we may assume that $|V(H)|=r+2$.
	If $2|E(H)| \geq r(r+2)$, then $\lambda_1(H) \geq r > \mu$.
	Therefore, we may assume that $2|E(H)| = r(r+2)-1$.
	
	Since every vertex of $H$ except one has degree at least $r$, $H$ is isomorphic to $K_1 \vee \overline{\frac{r-1}{2}K_2} \vee K_2$.
	Let $\{A,B,C\}$ be an equitable partition of $H$ such that $A$, $B$, and $C$ correspond to the vertex sets of $K_1$, $\overline{\frac{r-1}{2}K_2}$, and $K_2$ in $K_1 \vee \overline{\frac{r-1}{2}K_2} \vee K_2$, respectively.
	Then the quotient matrix of $H$ with respect to the partition $\{A,B,C\}$ is
	\begin{equation*}
		\begin{pmatrix}
			0 & r-1 & 0 \\
			1 & r-3 & 2 \\
			0 & r-1 & 1
		\end{pmatrix},
	\end{equation*}
	whose characteristic polynomial is equal to $p$.
	Thus, $\lambda_1(H) = \mu$.
\end{proof}

There are graphs achieving lower bounds of Lemma~\ref{lem:extreme}.
For positive integers $r$ and $\eta$ with $r > \eta$, we define a graph $H(r,\eta)$ as follows.
\begin{enumerate}[label=\rm(\roman*)]
	\item $H(r,\eta) = \overline{\lfloor \eta/2 \rfloor K_2} \vee K_{r+1- 2\lfloor \eta/2 \rfloor}$ if at least one of $r$ and $\eta$ is even.
	\item $H(r,\eta) = \overline{C_{\eta}} \vee \overline{\frac{r+2-\eta}{2} K_2}$ if both $r$ and $\eta$ are odd with $\eta \geq 3$.
	\item $H(r,\eta) = K_1 \vee \overline{\frac{r-1}{2} K_2} \vee K_2$ if $r$ is odd and $\eta=1$.
\end{enumerate}
When at least one of $r$ and $\eta$ is even, $H(r,\eta)$ has $r+1$ vertices and $\frac{r(r+1)- 2 \lfloor \eta/2 \rfloor}{2}$ edges.
When both $r$ and $\eta$ are odd with $\eta\geq 3$, $H(r,\eta)$ has $r+2$ vertices and $\frac{r(r+2)-\eta}{2}$ edges.
For odd $r$ with $r\geq 3$, $H(r,1)$ has $r+2$ vertices and $\frac{r(r+2)-1}{2}$ edges.
More precisely, only one vertex has degree $r-1$ and the other vertices have degree $r$ in $H(r,1)$.

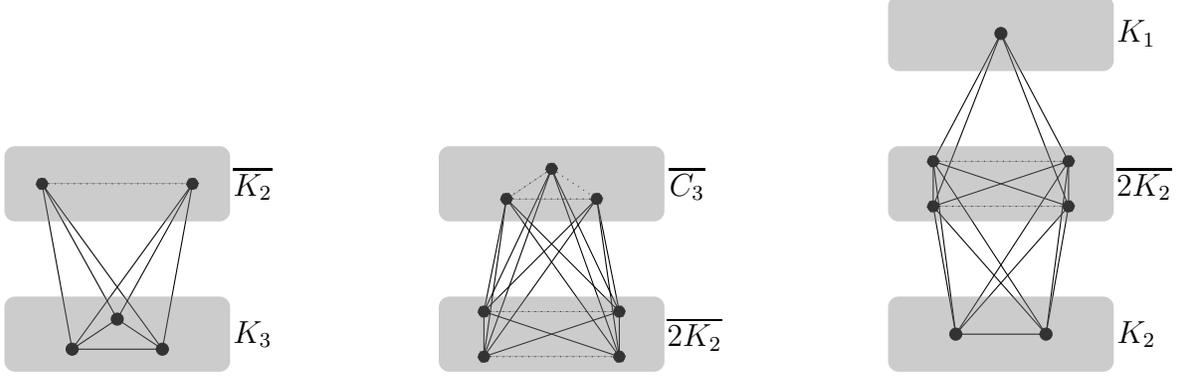
\begin{figure}[h!]
    \centering

    \begin{tikzpicture}
        \coordinate (a1) at (-1,2);
        \coordinate (a2) at (1,2);
        \coordinate (b1) at (0,0.2);
        \coordinate (b2) at (-0.6,-0.2);
        \coordinate (b3) at (0.6,-0.2);

		\node at (1.8,0) {$K_3$};
		\node at (1.8,2) {$\overline{K_2}$};
    
        \draw [dotted, fill] (a1) circle [radius=0.08] -- (a2) circle [radius=0.08];
        \draw [fill] (b1) circle [radius=0.08] -- (b2) circle [radius=0.08] -- (b3) circle [radius=0.08];
        \draw (b1) -- (b3);
        
        \foreach \i in {1,2} {
            \foreach \j in {1,2,3} {
                \draw (a\i) -- (b\j);
            }
        }
        
        \fill [fill=gray, opacity=0.4, rounded corners] (-1.5,-0.5) rectangle (1.5, 0.5);
        \fill [fill=gray, opacity=0.4, rounded corners] (-1.5,1.5) rectangle (1.5,2.5);
    \end{tikzpicture}
    \hspace{1.8cm}
    \begin{tikzpicture}
        \coordinate (a1) at (0,2.2);
        \coordinate (a2) at (-0.6,1.8);
        \coordinate (a3) at (0.6,1.8);
        \coordinate (b1) at (-0.9,0.3);
        \coordinate (b2) at (-0.9,-0.3);
        \coordinate (b3) at (0.9,0.3);
        \coordinate (b4) at (0.9,-0.3);

		\node at (1.9,0) {$\overline{2K_2}$};
		\node at (1.8,2) {$\overline{C_3}$};
    
        \draw [dotted, fill] (a1) circle [radius=0.08] -- (a2) circle [radius=0.08] -- (a3) circle [radius=0.08];
        \draw [dotted] (a1) -- (a3);
        
        \draw [dotted, fill] (b1) circle [radius=0.08] -- (b3) circle [radius=0.08];
        \draw [dotted, fill] (b2) circle [radius=0.08] -- (b4) circle [radius=0.08];
        \draw (b1) -- (b2);
        \draw (b1) -- (b4);
        \draw (b3) -- (b2);
        \draw (b3) -- (b4);
        
        \foreach \i in {1,2,3} {
            \foreach \j in {1,2,3,4} {
                \draw (a\i) -- (b\j);
            }
        }
        
        \fill [fill=gray, opacity=0.4, rounded corners] (-1.5,-0.5) rectangle (1.5, 0.5);
        \fill [fill=gray, opacity=0.4, rounded corners] (-1.5,1.5) rectangle (1.5,2.5);
    \end{tikzpicture}
    \hspace{1.8cm}
    \begin{tikzpicture}
        \coordinate (c) at (0,4);
        \coordinate (a1) at (-0.9,2.3);
        \coordinate (a2) at (-0.9,1.7);
        \coordinate (a3) at (0.9,2.3);
        \coordinate (a4) at (0.9,1.7);
        \coordinate (b1) at (0.6,0);
        \coordinate (b2) at (-0.6,0);

		\node at (1.8,0) {$K_2$};
		\node at (1.9,2) {$\overline{2K_2}$};
		\node at (1.8,4) {$K_1$};
        
        \draw [fill] (c) circle [radius=0.08];
    
        \draw [dotted, fill] (a1) circle [radius=0.08] -- (a3) circle [radius=0.08];
        \draw [dotted, fill] (a2) circle [radius=0.08] -- (a4) circle [radius=0.08];
        \draw (a1) -- (a2);
        \draw (a1) -- (a4);
        \draw (a3) -- (a2);
        \draw (a3) -- (a4);
        
        \draw [fill] (b1) circle [radius=0.08] -- (b2) circle [radius=0.08];
        
        \foreach \i in {1,2,3,4} {
            \draw (c) -- (a\i);
            \foreach \j in {1,2} {
                \draw (a\i) -- (b\j);
            }
        }
        
        \fill [fill=gray, opacity=0.4, rounded corners] (-1.5,-0.5) rectangle (1.5, 0.5);
        \fill [fill=gray, opacity=0.4, rounded corners] (-1.5,1.5) rectangle (1.5,2.5);
        \fill [fill=gray, opacity=0.4, rounded corners] (-1.5,3.5) rectangle (1.5,4.5);
    \end{tikzpicture}
    
    \caption{$H(4,2)$, $H(5,3)$, and $H(5,1)$ from left to right.}
    \label{fig:extremal}
\end{figure}

\begin{lem}
    For positive integers $r$ and $\eta$ with $r > \eta$, $\lambda_1(H(r,\eta)) = \rho(r,\eta)$. 
\end{lem}
\begin{proof}
	First, let us consider when at least one of $r$ and $\eta$ is even.
	Then $H(r,\eta) = \overline{\lfloor \eta/2 \rfloor K_2} \vee K_{r+1-2\lfloor \eta/2 \rfloor}$.
	If $\eta \leq 1$, then $\lambda_1(H(r,\eta)) = \lambda_1(K_{r+1}) = r$.
	Thus, we may assume that $\eta \geq 2$.
	Let $A = V(K_{r+1-2\lfloor \eta/2 \rfloor})$ and $B = V(\overline{\lfloor \eta/2 \rfloor K_2})$.
	Then $\{A,B\}$ is an equitable partition of $H(r,\eta)$.
	Let $Q$ be the quotient matrix of $G$ with respect to a partition $\{A,B\}$.
	Then
	\begin{align*}
		Q = \begin{pmatrix}
			r - 2 \lfloor \eta/2 \rfloor & 2 \lfloor \eta/2 \rfloor \\
			r + 1 - 2 \lfloor \eta/2 \rfloor & 2 \lfloor \eta/2 \rfloor - 2
		\end{pmatrix}.
	\end{align*}
	By Proposition~\ref{prop:equitable}, $\lambda_1(H(r,\eta)) = \lambda_1(Q)$.
	It is readily shown that
	\begin{equation*}
		\lambda_1(Q) = \frac{r-2 + \sqrt{r+2- 2 \lfloor \eta/2 \rfloor}}{2} = \rho(r,\eta).
	\end{equation*}

	When both $r$ and $\eta$ are odd with $\eta \geq 3$, let us take an equitable partition $\{V(\overline{C_{\eta}}), V(\overline{\frac{r+2-\eta}{2} K_2})\}$ of $H(r,\eta) = \overline{C_{\eta}} \vee \overline{\frac{r+2-\eta}{2} K_2}$.
	Then by imitating the preceding method, one can check that $\lambda_1(H(r,\eta)) = \rho(r,\eta)$.

	When $r$ is odd and $\eta=1$, let us take a partition $\{V(K_1), V(\overline{\frac{r-1}{2} K_2}), V(K_2)\}$ of $H(r,1) = \overline{K_1} \vee \overline{\frac{r-1}{2} K_2} \vee K_2$.
	Then by imitating the preceding method, one can check that $\lambda_1(H(r,\eta)) = \rho(r,\eta)$, which was already shown in the proof of Lemma~\ref{lem:extreme}-\ref{item:ext3}.
\end{proof}

\section{Proof of Theorem~\ref{thm:main}}\label{sec:main}

By using the Lov\'{a}sz's $(g,f)$-parity factor theorem and the interlacing theorem, we now prove our main result.

\begin{proof}[Proof of Theorem~\ref{thm:main}]
	Suppose that $G$ has no $(g,f)$-parity factor.
	By the Lov\'{a}sz's $(g,f)$-parity factor theorem, there are disjoint subsets $S$ and $T$ of $V(G)$ such that $q - f(S) + g(T) - d_G(T) - |[S,T]| > 0$, where $q$ is the number of components $C$ of $G \setminus (S\cup T)$ such that $f(V(C)) + |[T,V(C)]| \equiv 1 \pmod{2}$.
	Let $C_1, C_2, \dots, C_q$ be the components of $G \setminus (S\cup T)$ such that $f(V(C_i)) + |[T,V(C_i)]| \equiv 1 \pmod{2}$ for $i\in[q]$.

	Note that
	\begin{itemize}
		\item $g(T) \equiv f(T) \pmod{2}$,
		\item $d_G(T) + |[T,S]| \equiv |[T,V(G)\setminus (S\cup T)]| \pmod{2}$, and
		\item $q \equiv \sum_{C} ( f(V(C)) + |[T,V(C)]|) = f(V(G) \setminus (S\cup T)) + |[T, V(G) \setminus (S\cup T)]| \pmod{2}$, where the summation is over all components $C$ of $G \setminus (S\cup T)$.
	\end{itemize}
	Then $q - f(S) + g(T) - d_G(T) - |[S,T]| \equiv f(V(G)) \equiv 0 \pmod{2}$.
	Therefore,
	\begin{align*}
		2 &\leq
		q - f(S) + g(T) - d_G(T) + |[S,T]| \\
		&\leq
		q - \theta d_G(S) - (1-\theta) d_G(T) + |[S,T]| \\
		&=
		q
		- \theta \left( |[S,T]| + \sum_{i\in[q]} |[S,V(C_i)]| \right)
		- (1-\theta) \left( |[T,S]| + \sum_{i\in[q]} |[T,V(C_i)]| \right)
		+ |[S,T]| \\
		&=
		q
		- \theta \sum_{i\in[q]} |[S,V(C_i)]|
		- (1-\theta) \sum_{i\in[q]} |[T,V(C_i)]|. \tag{*}\label{ineq}
	\end{align*}

	\ref{item:m1}
	From the previous inequality~\eqref{ineq}, we deduce that
	\begin{align*}
		2
		&\leq
		q
		- \theta \sum_{i\in[q]} |[S,V(C_i)]|
		- (1-\theta) \sum_{i\in[q]} |[T,V(C_i)]| \\
		&\leq
		q - \theta^* \sum_{i\in[q]} |[S\cup T,V(C_i)]| \\
		&\leq
		q - \theta^* h q =
		q(1 - \theta^* h).
	\end{align*}
	If $h\geq 1/\theta^*$, then $2 \leq q(1-\theta^* h) \leq 0$, which is a contradiction.
	Therefore, we may assume that
	\begin{equation*}
		h < 1/\theta^* \leq \delta(G)
		\text{ and }
		\lambda_{\left\lceil \frac{2}{1-\theta^* h} \right\rceil}
	< \rho(\delta(G), \lceil 1/\theta^* \rceil -1).
	\end{equation*}
	Then $q \geq \frac{2}{1-\theta^* h}$ because $2\leq q(1-\theta^* h)$ and $1-\theta^* h > 0$.
  
	Let $x = \left\lceil \frac{2}{1-\theta^* h} \right\rceil$.
	We claim that at least $x$ of $C_1,\dots,C_q$ satisfy that $|[V(C_i), S\cup T]| < 1/\theta^*$.
	Suppose not.
	Then
	\begin{align*}
		2
		&\leq
		q - \theta^* \sum_{i\in[q]} |[S\cup T, V(C_i)]| \\
		&\leq
		q - (q-x+1) - \theta^* h(x-1)
		=
		(1 - \theta^* h)(x-1)
		<
		2,
	\end{align*}
	which is a contradiction.
	Therefore, the claim holds.
	By relabelling, we may assume that $|[V(C_i),S\cup T]| < 1/\theta^*$ for $i\in[x]$.

	Let us fix $i\in[x]$.
	We have $|[V(C_i),S\cup T]| \leq \lceil 1/\theta^* \rceil -1$ because $|[V(C_i),S\cup T]|$ is an integer.
	Since $|[V(C_i),S\cup T]| < 1/\theta^* \leq \delta(G)$, we have $|V(C_i)| \geq \delta(G)+1$ and
	\begin{equation*}
		2|E(C_i)| \geq \delta(G) |V(C_i)| - |[V(C_i),S\cup T]| \geq \delta(G) |V(C_i)| - (\lceil 1/\theta^* \rceil - 1).
	\end{equation*}
	By Lemma~\ref{lem:extreme}, $\lambda_1(C_i) \geq \rho(\delta(G), \lceil 1/\theta^* \rceil -1)$.
	By the interlacing theorem,
	\[
	    \lambda_x(G) \geq \lambda_x(G \setminus (S\cup T)) \geq \min\{\lambda_1(C_j) : j\in[x]\} \geq \rho(\delta(G), \lceil 1/\theta^* \rceil -1),
    \]
    which is a contradiction.
  
	\ref{item:m2}
	Since $d_G(v)$ is even for all vertex $v$ in $G$, for each $i\in[q]$, we have $|[S\cup T,V(C_i)]| = d_G(V(C_i)) - 2|E(C_i)| \equiv 0 \pmod{2}$.
	Since $f(v)$ is even for all vertex $v$ in $G$, for each $i\in [q]$, we have $|[T,V(C_i)]| \equiv |[T,V(C_i)]| + f(V(C_i)) \equiv 1 \pmod{2}$.
	Therefore, both $|[S,T(C_i)]|$ and $|[T,V(C_i)]|$ is at least $1$ for each $i\in[q]$.
	By the inequality~\eqref{ineq},
	\begin{equation*}
		2 \leq
		q
		- \theta \sum_{i\in[q]} |[S,V(C_i)]|
		- (1-\theta) \sum_{i\in[q]} |[T,V(C_i)]|
		\leq 0,
	\end{equation*}
	which is a contradiction.
  
	\ref{item:m3}
	Since $d_G(v)$ is even for all vertex $v$ in $G$, for each $i\in[q]$, we have $|[S\cup T,V(C_i)]| = d_G(V(C_i)) - 2|E(C_i)| \equiv 0 \pmod{2}$, which implies that $|[S\cup T,V(C_i)]| \geq h_e$.
	By the inequality~\eqref{ineq},
	\begin{align*}
		2
		\leq
		q - \theta^* \sum_{i\in[q]} |[S\cup T, V(C_i)]|
		\leq
		q(1 - \theta^* h_e q).
	\end{align*}
	If $h_e \geq 1/\theta^*$, then $2 \leq q(1-\theta^* h_e) \leq 0$, which is a contradiction.
	Therefore, we may assume that $h_e < 1/\theta^* \leq \delta(G)$ and $\lambda_{\left\lceil \frac{2}{1-\theta^* h_e} \right\rceil} < \rho(\delta(G), \lceil 1/\theta^* \rceil -1)$.
	Then $q \geq \frac{2}{1-\theta^* h_e}$.
  
	The remaining proof is similar to the proof of~\ref{item:m1}, so we omit details.

	\ref{item:m4}
	Since $f(v)$ is even for all vertex $v$ in $G$, for each $i\in [q]$, we have $|[T,V(C_i)]| \equiv |[T,V(C_i)]| + f(V(C_i)) \equiv 1 \pmod{2}$, which implies that $|[T,V(C_i)]| \geq 1$.
	Let $q'$ be the number of components $C_i$ such that $|[S,V(C_i)]| = 0$.
	By relabelling, we may assume that $|[S,V(C_i)]| = 0$ for $i\in[q']$.
	Note that $|[T,V(C_i)]| \geq h_o$ for each $i\in[q']$.
	By the inequality~\eqref{ineq},
	\begin{align*}
	  2
	  &\leq
	  q
	  - \theta \sum_{i\in[q]} |[S,V(C)]|
	  - (1-\theta) \sum_{i\in[q]} |[T,V(C)]| \\
	  &\leq
	  q'
	  - (1-\theta) \sum_{i\in[q']} |[T,V(C)]| \\
	  &\leq
	  q' - (1-\theta) h_o q' = q'(1-(1-\theta)h_o).
	\end{align*}
	If $h_o > 1/(1-\theta)$, then $2\leq q'(1-(1-\theta)h_o) \leq 0$, which is a contradiction.
	Therefore, we can assume that $h_o < 1/(1-\theta) \leq \delta(G)$ and $\lambda_{\left\lceil \frac{2}{1-(1-\theta) h_o} \right\rceil} < \rho(\delta(G), \lceil 1/(1-\theta) \rceil -1)$.
	Then $q' \geq \frac{2}{1-(1-\theta)h_o}$.

	The remaining proof is similar to the proof of~\ref{item:m1}, so we omit details.

	\ref{item:m5}
	Let $g'$ and $f'$ be integer-valued functions on $V(G)$ such that $g'(v) = d_G(v) - f(v)$ and $f'(v) = d_G(v) - g(v)$.
	Then $g'(v) \equiv f'(v) \equiv 0 \pmod{2}$ and $g'(v) \leq (1-\theta) d_G(v) \leq f'(v)$ for all $v\in V(G)$.
	Note that $G$ has a $(g,f)$-parity factor if and only if it has a $(g',f')$-parity factor.
	Therefore, applying~\ref{item:m3} for $g'$ and $f'$, we conclude~\ref{item:m5}.
\end{proof}

\section{Tight examples}\label{sec:tight}
To show that the bounds in Theorem~\ref{thm:main} are sharp, we bring some definitions from~\cite{OW2011} and slightly modify them.

\begin{defn}\label{def:splicing}
Let $B$ be a graph with $\Delta(B) = a$ such that $\sum_{v\in V(B)} (a - d_B(v)) = b.$ If $u$ is a vertex of degree $b$ or $b+1$ in a graph $H$, then \emph{splicing} $B$ into $u$ means deleting $u$ and replacing each edge of the form $uw$ in $H$ with an edge from $w$ to a vertex of $B$, distributed so that each vertex of B now, except possibly only one vertex with degree $a+1$, has degree~$a$.
\end{defn}

It is readily shown that, from Definition~\ref{def:splicing}, the edge-connectivity of a graph obtained from $H$ by splicing $B$ into $u$ is at least the minimum of $\kappa'(B)$ and $\kappa'(H)$.

\begin{defn}\label{def:F}
For integers $r,h,l$ with $l \ge r > h \ge 1$, let $\mathcal{F}_{r,h,l}$ be the set of graphs obtained from the complete bipartite graph $K_{h,l}$ by splicing $H(r,h)$ into each vertex having degree $h$ in $K_{h,l}$.
For each $F \in \mathcal{F}_{r,h,l}$, we denote by $U(F)$ the set of vertices in $F$, which corresponds to the one of vertices with degree $l$ in $K_{h,l}$.
\end{defn}

\begin{obs}
	Every graph $F$ in $\mathcal{F}_{r,h,l}$ from Definition~\ref{def:F} satisfies the following:
\begin{itemize}%
    \item Every vertex of $U(F)$ has degree $l$ in $F$, and $F - U(F)$ is isomorphic to $l H(r,h)$.
    \item $|V(H(r,h))|$ is even if and only if $r$ is odd and $h$ is even.
    \item If $r$ is odd or $h$ is even, then every vertex of $V(F)\setminus U(F)$ has degree $r$ in $F$.
    \item If $r$ is even and $h$ is odd, then exactly one vertex of each component $C$ of $F- U(F)$ has degree $r+1$ in $F$, and every other vertex of $C$ has degree $r$ in $F$.
    \item If $r=l$, and $r$ is odd or $h$ is even, then $F$ is $r$-regular.
\end{itemize}
\end{obs}

\begin{thm}\label{thm:tight}
	Let $r,h,l$ be integers with $l \ge r > h \ge 1$, and $F$ be a graph in $\mathcal{F}_{r,h,l}$.
	Then
	\begin{enumerate}[label=\rm(\roman*)]
		\item $\delta(F)=r$,
		\item $\kappa'(F)=h$, and
		\item $\lambda_{h+1}(F) = \cdots = \lambda_l(F)=\rho(r,h)$.
        \item %
        Let $b$ be the largest odd integer with $b < l/h$, and let $g$ and $f$ be integer-valued function on $V(F)$ such that $g(v) \leq f(v)$ and $g(v) \equiv f(v) \pmod{2}$ for each $v\in V(F)$, and $f(w) = b$ for each $w \in U(F)$, and $f(V(C)) \equiv 1 \pmod{2}$ for each component $C$ of $F - U(F)$.
	Then $F$ has no $(g,f)$-parity factor.
		In particular, $F$ has no odd $[1,b]$-factor whenever $r$ is even or $h$ is odd.
    \end{enumerate}
\end{thm}

\begin{proof}
(i) By the construction of $F$, every vertex has degree $r, r+1,$ or $l$. Thus we have the desired result.

\vskip 0.05in

\noindent
(ii)
One may check that $\kappa'(K_{h,l}) = h$ and $\kappa(H(r,h)) = r-1$, which imply that $\kappa'(F) \geq h$.
For a component $C$ of $F- U(F)$, we have $|[V(C),U(F)]_F| = h$.
Therefore, $\kappa'(F) = h$.

\vskip 0.05in
\noindent
(iii) Since there are $l$ disjoint copies of $H(r,\eta)$ in $F$, by Theorem~\ref{thm:interlacing}, we have
$$\lambda_l(F) \ge \lambda_l(lH(r,h)) = \rho(r,\eta)=\lambda_1(lH(r,h))
\ge \lambda_{|V(F)|-(|V(F)|-h)+1}(F) = \lambda_{h+1}(F),$$
which gives $\lambda_{h+1}(F)= \cdots=\lambda_l(F)=\rho(r,h)$.

\vskip 0.05in
\noindent
(iv) Assume to the contrary that $F$ has a $(g,f)$-parity factor. Applying Theorem~\ref{thm:parityfactor} for $S= U(F)$ and $T=\emptyset$, we have
\begin{align*}
    0
    \le d_G(\emptyset) - g(\emptyset) + f(U(F)) - |[U(F),\emptyset]_F| - q
    = bh - q,
\end{align*}
where $q$ is the number of components $C$ in $F\setminus U(F)$ such that $f(V(C)) + |[\emptyset,V(C)]_F|$ is odd.
Then $q = l$ so $0\leq bh - l < 0$, which is a contradiction.
\end{proof}

Now, we explain why Theorem~\ref{thm:tight} makes the eigenvalue conditions of Theorem~\ref{thm:main} tight.

For an integer $r$ with $r\geq 2$ and a real $\theta$ with $1/r \le \theta \le 1/2$, let $h$ and $l$ be integers such that $h = \lceil 1/\theta \rceil - 1 < 1/\theta$ and $l \geq 2h/(1-\theta h)$.
Then $h \geq 1/\theta -1$, so $1/(1-\theta h) \geq 1/\theta$.
Therefore, $h+1 = \lceil 1/\theta \rceil \leq \lceil 2/(1-\theta h) \rceil \leq l$.
By Theorem~\ref{thm:tight}, we obtain a graph $G$ such that
\begin{enumerate}[label=\rm(\roman*)]
	\item $\delta(G) = r$,
	\item $\kappa'(G) = h < 1/\theta$,
	\item $\lambda_{\left\lceil \frac{2}{1-\theta h} \right\rceil} = \rho(r,h)$, and
	\item there are maps $g,f:V(G) \to \mathbb{Z}$ such that $g(v) \leq \theta d_G(v) \leq f(v)$ and $g(v) \equiv f(v) \pmod{2}$ for all $v\in V(G)$, and $G$ has no $(g,f)$-parity factor.
\end{enumerate}
This implies that the eigenvalue condition of Theorem~\ref{thm:main}-\ref{item:m1} is optimal.
If we additionally assume that $r$ and $l$ are even with $l \geq \frac{2}{1-\theta(h+1)}$, then the previous graph $G$ is a tight example for Theorem~\ref{thm:main}-\ref{item:m3}.
Suppose that $r$, $h$, and $l$ are odd with $l \geq \frac{2}{1-\theta(h+1)}$.
Then each vertex of $G$ has odd degree, and $|V(H(r,h))|$ is odd.
By (iv) of Theorem~\ref{thm:tight}, $G$ has no odd $[1,b]$-factor, where $b$ is the largest odd integer less than $l/h$.
Furthermore, $G$ has no $(g',f')$-parity factor where $g',f': V(G) \to \mathbb{Z}$ such that $g'(v) = d_G(v)-b$ and $f'(v) = d_G(v)-1$ for each $v\in V(G)$.
Therefore, the eigenvalue conditions of Theorem~\ref{thm:main}-\ref{item:m4} and~\ref{item:m5} are sharp.

\begin{rmk}
    By Theorem~\ref{thm:tight}, we could show that Theorem~1.3 in the paper~\cite{O2022} is sharp.
    Let $r$, $h$, and $b$ be positive integers with $r>h \geq 2$ and $b\equiv 1 \pmod{2}$.
    Suppose that
    \[
        r \equiv h \text{ and } \max\{b,h\} < \left\lceil \frac{2r}{r-bh} \right\rceil = r.\tag{**}
    \]
    Take $l = r$.
    Then applying Theorem~\ref{thm:tight}, we obtain an $r$-regular $h$-edge-connected graph $G$ such that $\lambda_r(G) = \rho(r,h)$ and $G$ has no odd $[1,b]$-factor.
    If $r$ is even, then $G$ has no odd $[r-b,r-1]$-factor, and if $r$ is odd, then $G$ has no even $[r-b,r-1]$-factor.
    For example, $(r,h,b) \in \{ (4k,2,2k-1), (2k+2,2k,3), (6k-1,3,2k-1), (6k-1,2k-1,3) \}$ with $k\geq3$ satisfies the condition~(**).
    This is a sharp example of Theorem~1.3 in O~\cite{O2022}.
\end{rmk}

\section*{Acknowledgements}
The second author would like to thank Dr. Sang-il Oum for fruitful discussions and also for providing a wonderful research environment at Institute for Basic Science.

\providecommand{\bysame}{\leavevmode\hbox to3em{\hrulefill}\thinspace}
\providecommand{\MR}{\relax\ifhmode\unskip\space\fi MR }
\providecommand{\MRhref}[2]{%
  \href{http://www.ams.org/mathscinet-getitem?mr=#1}{#2}
}
\providecommand{\href}[2]{#2}

\end{document}